\documentclass[11pt]{article}

\usepackage{amssymb,amsmath,amsthm}
\usepackage{amsfonts}
\usepackage{latexsym}
\usepackage{enumerate}
\usepackage{wasysym}
\usepackage{fancyhdr}
\usepackage{lipsum}
\usepackage{epsfig}
\usepackage{breqn}
\usepackage{footnpag}
\usepackage{rotating}
\usepackage{amsfonts}
\usepackage{setspace}
\usepackage{fullpage}
\usepackage{enumitem}
\usepackage{bbold} 
\usepackage{comment}
\usepackage{pgf,tikz}
\usepackage{hyperref}
\usepackage{verbatim}
\usepackage{xcolor}
\usepackage{mathrsfs}
\usepackage{soul}
\usetikzlibrary{arrows}
\usepackage{array}

\newtheoremstyle{dotless}{}{}{\itshape}{}{\bfseries}{}{ }{}

\theoremstyle{dotless}

\newtheorem{theorem}{Theorem}[section] 
\newtheorem{proposition}[theorem]{Proposition}

\newtheorem{lemma}[theorem]{Lemma}

\newtheorem{observation}[theorem]{Observation}

\newtheorem*{theoremsp1}{Theorem \ref{thm:turan}}

\newcommand{\FL}[1]{\left\lfloor #1 \right\rfloor}
\newcommand{\CL}[1]{\left\lceil #1 \right\rceil}

\renewcommand{\l}{\ell}

\newcommand{\mc}[1]{\mathcal{#1}}

\newcommand{\maxcr}{\mathrm{max-\overline{cr}}}

\DeclareMathOperator{\ex}{ex}

\title{Weighted Tur\'an Problems with Applications}
\author{ {\Large Patrick Bennett}\thanks{\url{patrick.bennett@wmich.edu}, supported in part by Simons Foundation Grant \#426894.} \\Department of Mathematics \\ Western Michigan University \and {\Large Sean English}\thanks{\url{Sean.English@ryerson.ca}} \\Department of Mathematics \\ Ryerson University \and  \\{\Large Maria Talanda-Fisher}\thanks{\url{maria.talanda-fisher@wmich.edu}} \\Department of Mathematics \\ Western Michigan University} 

\begin{document}
	\maketitle
	
	\begin{abstract}
	Suppose the edges of $K_n$ are assigned weights  by a weight function $w$. We define the {\em weighted extremal number}
	\[
	\ex(n,w,F):=\max\{w(G)\mid G\subseteq K_n,\text{ and }G\text{ is }F\text{-free}\}\] 
	where $w(G):=\sum_{e\in E(G)}w(e)$. In this paper we study this problem for two types of weights $w$, each of which has an application. The first application is to an extremal problem in a complete multipartite host graph. The second application is to the maximum rectilinear crossing number of trees of diameter 4. 
	\end{abstract}

	\section{Introduction}

	The first known result in extremal graph theory is Mantel's Theorem, \cite{M07}, which states that the maximum number of edges in a triangle-free graph on $n$ vertices is realized by the balanced complete bipartite graph $K_{\lceil{n/2}\rceil,\lfloor{n/2}\rfloor}$. This was later generalized to graphs that contain no copy of $K_\ell$ by Tur\'an, \cite{T41}. Since then, the study of such problems has been at the forefront of extremal graph theory.
	
	We say a graph $G$ is $F$-free if $G$ does not contain a subgraph isomorphic to $F$. The \emph{extremal number}, or \emph{Tur\'an number} of a graph $F$, denoted $\ex(n,F)$, is the maximum number of edges over all $n$-vertex graphs that are $F$-free. The \emph{Tur\'an graph}, $T_{n,\ell}$, is the balanced (i.e. partite sets differ in size by at most one) complete $\ell$-partite graph on $n$ vertices. This allows us to state the result by Tur\'an as follows:
	
	\begin{theorem}[Tur\'an's Theorem,\cite{T41}]
		For all $n\geq\ell\geq 3$, we have
		\[
		\ex(n,K_\ell)=|E(T_{n,\ell-1})|.
		\]
	\end{theorem}
	
	Extremal numbers have been a main topic of interest in extremal graph theory since the topic's inception. Many variants on the classical extremal number question have been studied. One variation involves changing the host graph from $K_n$ to some other graph $G$. More precisely, the extremal number $\ex(H,F)$ is the largest number of edges over all $F$-free subgraphs of $H$. Note, $\ex(K_n,F)=\ex(n,F)$.
	
	This variation originated with Zarankiewicz, who was interested in the case where the host graph is bipartite, specifically $\ex(K_{n,n},K_{s,s})$, \cite{Z51}. More recently, De Silva, Heysse, Kapilow, Schenfisch and Young determined the numbers $\ex(K_{a_1,a_2,\dots,a_\ell},s K_\ell)$, forbidding a union of disjoint cliques in a complete multi-partite host graph, \cite{DY}. The authors suggest an open problem; determining $\ex(K_{k_1,k_2,\dots,k_r},s K_\ell)$, where the number of partite sets in the host graph $k\geq r$, the size of the forbidden cliques. Towards this end, we determine these numbers exactly in the case $s=1$. For a set of indices $P\subseteq [r]$, let $k_P=\sum_{i\in P} k_i$.
	
	\begin{theoremsp1}
	The extremal number $\ex(K_{k_1, k_2, \ldots, k_r}, K_\l)$ is equal to 
	\[
	\max_{\mc{P} } \sum_{\substack{P, P' \in \mc{P} \\ P \neq P'}}k_{P}\cdot k_{P'}
	\]
	where the max is over all partitions $\mc{P}$ of $[r]$ into $\l-1$ sets.
	\end{theoremsp1}
	
	In order to prove Theorem \ref{thm:turan}, we reduced the problem  to a weighted version of Tur\'ans theorem, where the vertices of the host graph receive weights given by vertex-weight function $W$, and then the product-edge-weight function $w_\Pi$ is given by the product of the weights of the endpoints. The extremal number $\ex(n,w_\Pi, F)$ is the maximum sum of edge-weights over all $F$-free graphs on $n$ vertices with edge-weights given by $w_\Pi$. The numbers
	$\ex(n,w_\Pi,K_\ell)$ are determined exactly for all non-negative vertex-weight functions $W$ in Section \ref{section prodedgeweight}.
	
	In addition to the product edge-weighting, we explore the min-edge-weight function $w_{min}$, where given a vertex-weight function $w$, the edge-weight function, $w_{min}$, is given by the minimum weight of the two incident vertices. In section \ref{section minedgeweight}, the numbers $\ex(n,w_{min},K_\ell)$ are determined exactly for any vertex-weight function $w$.
	
	As an application of the min-edge-weight function extremal numbers, we explore a variation of crossing numbers. The most common question about crossing number explores the minimum number of crossings over all planar drawings of a non-planar graph $G$. Under some natural assumptions, one can also explore the maximum number of crossings possible over all planar drawings of a graph $G$. This is known as the \emph{maximum crossing number}. Here we explore \emph{maximum rectilinear crossing numbers} of $G$, denoted $\maxcr(G)$, where the edges are required to be drawn as straight line segments between their incident vertices.
	
	Recently, Fallon, Hogenson, Keough, Lomel\'i Schaefer and Sober\'on determined a lower bound on the maximum rectilinear crossing number of spiders (trees with a single vertex of degree $\geq 3$). In section \ref{section spiders}, we provide a matching upper bound. Finally, using min-edge-weight extremal numbers, we determine the maximum rectilinear crossing number of trees of diameter $4$ in section \ref{section diameter 4}.

	\section{Weighted Tur\'an}	
	
	Let $w:E(K_n)\to [0,\infty)$ be a nonnegative-valued weight function of the edges of $K_n$. For any $G \subseteq K_n$ let $w(G):=\sum_{e\in E(G)}w(e)$. Given a graph $F$, we define the {\em weighted extremal number}
	\[
	\ex(n,w,F):=\max\{w(G)\mid G\subseteq K_n,\text{ and }G\text{ is }F\text{-free}\}.\] 
Of course if $w$ is the constant function $1$ then the we just get the standard extremal number $\ex(n, F)$.

	In this paper we consider two types of weight functions $w$, both of which can be described as {\em vertex-induced edge weightings}, in that they naturally arise from a weight function on the vertices. Let $W:V(K_n)\to [0,\infty) $ be a vertex weighting. Let the min-edge-weight $w_{\min}:E(K_n)\to [0,\infty)$ be given by $w_{\min}(uv)=\min\{W(u),W(v)\}$. Similarly, let the product-edge-weight $w_\Pi:E(K_n)\to[0,\infty)$ be given by $w_\Pi(uv)=W(u)W(v)$. In this paper we concern ourselves with $\ex(n,w_{min},K_\ell)$ and $\ex(n,w_\Pi,K_\ell)$.

	\subsection{Product-edge-weight}\label{section prodedgeweight}
	
 For a set of vertices $S$, let $W(S):= \sum_{v \in S} W(v)$.

	\begin{proposition}\label{prop:prod}
		The extremal number $\ex(n,w_\Pi,K_\ell)$ is equal to 
\[
\max_{\mc{P}}\sum_{\substack{P, P' \in \mc{P} \\ P \neq P'}} W(P) W(P')
\]
where the maximum is taken over all partitions $\mc{P}$ of $V(K_n)$ into $\l-1$ parts. 
	\end{proposition}

Note that the expression inside the ``max" is precisely the weight of the complete $(\l-1)$-partite graph with partition $\mc{P}$. Thus, the lemma is equivalent to saying that there exists a complete $(\l-1)$-partite graph that is extremal. Here we prove the stronger result that if all vertex weights are strictly positive, then actually any extremal example must be complete-$(\ell-1)$-partite.

\begin{proof}
We will use a proof technique first introduced by Zykov \cite{Zykov} to prove a generalization of Tur\'an's theorem. This technique gives a particularly elegant proof of Tur\'an's theorem which is explained very clearly in \cite{book}.

 Let $x, y \in V$ be nonadjacent in a graph $G$. If we form the graph $G'$ by deleting $y$ and then adding a new vertex $x'$ of weight $W(y)$ such that $N_{G'}(x') = N_G(x)$. We call this operation {\it duplicating the vertex $x$ to replace $y$} and we call the new vertex $x'$ a {\it duplicate} of $x$ since they have the same neighborhood. Note that if $G$ is $K_\l$-free then so is $G'$. 

Suppose that $G$ has maximum possible weight among $K_\l$-free graphs. Vertices of weight $0$ cannot affect the value of $\ex(n,w_\Pi,K_\ell)$, so we will assume every vertex has strictly positive weight. Note that if $W(N_G(x)) > W(N_G(y))$, then duplicating $x$ to replace $y$ with $x'$, where $x'$ receives the same weight that $y$ had, yields a new graph $G'$ whose weight is more than that of $G$. Thus we may assume that any pair of nonadjacent vertices in $G$ have neighborhoods of the same weight. 

Recall that a graph is complete multipartite if and only if it does not have any three vertices with exactly one edge among them. Suppose for the sake of contradiction that $G$ has three vertices $x, y, z$ such that $yz \in E(G)$ but $xy, xz \notin E(G)$. Then we must have $ W(N_G(x)) = W(N_G(y)) = W(N_G(z))$.

Now we replace $y$ and $z$ with duplicates of $x$. We will show that the new graph $G'$ has a larger weight than $G$, which will give us a contradiction. Indeed, when we lose the edges incident to $y$ and $z$, we lose a total weight of 
\[
W(y)\cdot W(N(y)) + W(z) \cdot W(N(z)) - W(y)W(z),
\] 
and then when we add the vertices $x'$ (whose weight will be $W(y)$) and $x''$ (of weight $W(z)$) we gain a total of
\[
W(y)\cdot W(N(x)) + W(z) \cdot W(N(x))
\]
which is more than what we lost. Thus $G$ is not extremal, a contradiction, so $G$ must be complete multipartite. Then $G$ must be complete $(\ell-1)$-partite, since if $G$ had $\l$ nonempty parts then $G$ would contain a copy of $K_{\ell}$, and with $k\leq \ell-2$, $G$ could not be extremal. This completes the proof.

\end{proof}

	\subsection{Min-edge-weight}\label{section minedgeweight}

	In this section we address the edge-weight function $w_{min}$, induced by the vertex-weight function $W$, given by $w_{min}(xy) = \min\{W(x), W(y)\}$. We define the following notation.  Assume (WLOG) that $W(v_n) \le \ldots \le W(v_1)$. Let $\mathcal{F}(n, W ,t)$ be the family of $K_\l$-free graphs on vertex set $\{v_1,\dots,v_n\}$ with vertex-weight function $W$ with $t$ edges that maximize $w_{min}(G)$ (among all such graphs $G$). Let $B_{\ell}(v_1, \ldots, v_n)$ be the complete $(\l-1)$-partite graph on vertex set $\{v_1, \ldots, v_n\}$ whose partition $\mc{P}$ is given by putting each $v_i$ into part $P_{i \pmod{\l-1}}$. We find it convenient to use congruence classes as indices but we will abuse notation and write $P_i$ instead of $P_{i \pmod{\l-1}}$ for convenience. 
	
	Let $t_{n, \l}= |E(T_{n, \l})|$, where $T_{n, \l}$ is the Tur\'an graph, or the complete $\l$-partite graph on $n$ vertices where the parts are balanced, i.e. they are all either of size $\FL{\frac n\l}$ or $\CL{\frac n\l}$. Note that since we can form the graph $T_{n, \l}$ by starting with $T_{n-1, \l}$ and inserting one new vertex into a part containing $\FL{\frac{n-1}{\l}}$ vertices (and making the new vertex adjacent to everything outside of the part it is in) we have
	\begin{equation}\label{eq:tnl}
	t_{n, \l} =t_{n-1,\l}+ n-1-\FL{\frac{n-1}{\l}}.
	\end{equation}

	\begin{proposition}\label{proposition minedgeweight}
	 For all $n \ge 1$, $0 \le t \le t_{n, \l-1}$ and vertex-weight functions $W$ such that $W(v_n) \le \ldots \le W(v_1)$, there exists a graph $B \subseteq B_\ell(v_1, \ldots, v_n)$ such that $B \in \mc{F}(n, W, t)$. Consequentially,
	 \[
	 \ex(n,w_{min},K_\ell)=w_{min}(B_\ell(v_1,\dots,v_n)).
	 \]
	\end{proposition}
	
	\begin{proof}
	 We proceed via induction on $n$. Our base case is $n=1$, so the graph has no edges and so the claim is true. Now for the induction step, let $0 \le t \le  t_{n, \l-1}$ and $G$ be a graph with $t$ edges that has maximum weight among $K_\ell$-free graphs. Let $d_G(v_n)=d$. We are guided by the intuition that $d$ should not be too large since edges incident to $v_n$ have the smallest possible weight.  We apply our induction hypothesis to the graph $G-v_n$, so there exists some graph $G'\subseteq B_\ell(v_1, \ldots, v_{n-1})$ with $t-d$ edges and $w_{min}(G')\ge  w_{min}(G-v_n)$. 
	 
	 Let $y=\text{max}\{0, t-t_{n-1, \l-1} \}$. Let $G''$ be a graph with $t-y\leq t_{n-1,\l-1}$ edges such that $G'\subseteq G''\subseteq B_{\ell}(v_1,\dots,v_{n-1})$. Now form the graph $B$ by adding $v_n$ into $G''$ such that $d_{B}(v_n)=y$ and such that $v_n$ has no neighbors in the part $P_{n}$ of the partition of $G''$ (recall that as a subgraph of $B_\ell(v_1,\dots,v_{n-1})$, $G''$ has an $(\l-1)$-partition and the parts are indexed by congruence classes mod $\l-1$). Choosing $B$ such that $P_{n} \cap N_{B}(v_n)=\emptyset$ is possible since $y \le t_{n, \l-1}-t_{n-1,\l-1} = n-1-\FL{\frac{n-1}{\l-1}}$ by \eqref{eq:tnl}, and in $G'$ the part $P_n$ has only $\FL{\frac{n-1}{\l-1}}$ vertices. Thus $B \subseteq B(v_1, \ldots, v_n)$, and since $w_{min}(G')\ge w_{min}(G-v_1) = w_{min}(G)-dW(v_1)$, we have 
	 \[
	 w_{min}(B)\ge w_{min}(G')+W(v_2)(d-y)+W(v_1)y \ge w_{min}(G) + (W(v_2)-W(v_1))(d-y) \ge w_{min}(G)
	 \]
	 so $B\in \mc{F}(n, w, t)$.
	 
	 Thus, any extremal example will have the same weight as some subgraph of $B_\ell(v_1,\dots,v_n)$. Since $B_\ell(v_1,\dots,v_n)$ is $K_\ell$-free and has weight at least as great as any subgraph of itself, $B_\ell(v_1,\dots,v_n)$ must be an extremal example, finishing the proof.
	 \end{proof}

	\section{Multipartite Tur\'an}
In this section we address a problem suggested by De Silva, Heysse, Kapilow, Schenfisch, and Young \cite{DY}. For general graphs $H, F$ we  define the {\em extremal number of $F$ with host graph $H$}, $\ex(H, F)$, to be the largest number of edges in any graph $G$ such that $F \not\subseteq G \subseteq H$. Of course when $H=K_n$ we just get the standard extremal number $\ex(n, F)$. In \cite{DY} they determine the value of $\ex(K_{k_1, k_2, \ldots, k_r}, j K_r)$, the case where the host graph $H$ is complete $r$-partite and the forbidden graph $F$ consists of $j$ vertex-disjoint $r$-cliques. It is natural to ask for the value of $\ex(K_{k_1, k_2, \ldots, k_r}, j K_\l)$ for all $r \ge \l$, but the proof techniques in \cite{DY} do not seem to generalize here. In this section we make some progress on this problem: we handle all $r \ge \l$ but only for $j=1$. Unfortunately it seems it would require significant new ideas to handle all $j$. 

For a set of indices $P$ define $k_P:=\sum_{i \in P} k_i$.	
	\begin{theorem}\label{thm:turan}
		The extremal number $\ex(K_{k_1, k_2, \ldots, k_r}, K_\l)$ is equal to 
\[
\max_{\mc{P}}\sum_{\substack{P, P' \in \mc{P} \\ P \neq P'}} k_P k_{P'}
\]
where the max is over all partitions $\mc{P}$ of $[r]$ into $\l-1$ sets. 
	\end{theorem}
	
	Note that the expression in the ``max" above is precisely the number of edges in the complete $(\l-1)$-partite subgraph $G$ of $K_{k_1, k_2, \ldots, k_r}$ formed by merging parts of the partition $\mc{P}$. In other words, for each part $P \in \mc{P}$ the graph $G$ has a part $A_P$ consisting of $k_P$ many vertices. Thus, theorem \ref{thm:turan} states that there exists an extremal $(\l-1)$-partite graph.

\begin{proof}
Suppose $G$ is extremal, in other words $G$ has the largest possible number of edges among all subgraphs of $K_{k_1, k_2, \ldots, k_r}$ that do not contain $K_\l$. Since $G\subseteq K_{k_1, \ldots, k_r}$, $G$ admits an $r$-partition, $V(G)=A_1 \cup \ldots \cup A_r$, with $|A_i|=k_i$. Note that if $x$ and $y$ are in the same partite set $A_i$, then if we duplicate $x$ to replace $y$, the graph we obtain, $G'$, is still a subgraph of $K_{k_1, k_2, \ldots, k_r}$. We will apply the duplicating operation a bunch of times. We do not want to call the new graphs $G', G'', G''', \ldots$ so instead we will abuse notation by referring to a single graph $G$ which changes each time we do an operation.

We choose the highest degree vertex $x_1 \in A_1$, and then for every other vertex $x \in A_1$ we duplicate $x_1$ to replace $x$. After these duplications, every vertex in $G$ is either adjacent to all of $A_1$ or none of $A_1$, and the number of edges in $G$ could not have gone down. Now we choose the higest degree vertex $x_2 \in A_2$ and  for every other vertex $x \in A_2$ we duplicate $x_2$ to replace $x$. This preserves the property that every vertex is either adjacent to all of $A_1$ or none of $A_1$; additionally, now every vertex is either adjacent to all of $A_2$ or none of $A_2$. We proceed in this fashion, making all vertices in $A_i$ duplicates of each other for all $i \in [k]$. At the end of the process, the graph $G$ must have at least as many edges as when we started, and has the property that for all $1\leq i,j\leq r$, either all the edges between $A_i$ and $A_j$ are present, or none are.

The rest of the proof follows from Proposition \ref{prop:prod}: Consider the auxiliary graph $H$, where the vertices of $H$ are the sets $A_i$, and $A_iA_j$ is an edge in $H$ if all the edges between $A_i$ and $A_j$ are present in $G$. Let $w$ be a vertex-weight function on $H$ where $w(A_i)=|A_i|=k_i$. Then, the product-edge-weight function, $w_\Pi$, has the property that $w_\Pi (H)=|E(G)|$. Clearly, if $H$ had a copy of $K_\ell$, $G$ would also have a copy, so $H$ is $K_\ell$-free. Thus by Proposition \ref{prop:prod},
\[
|E(G)|=w_\Pi(H)\leq \ex(n,w_\Pi,K_\ell)=
\max_{\mc{P}}\sum_{\substack{P, P' \in \mc{P} \\ P \neq P'}} W_P W_{P'}
=
\max_{\mc{P}}\sum_{\substack{P, P' \in \mc{P} \\ P \neq P'}} k_P k_{P'}.
\]

\end{proof}
	
	\section{Maximum Rectilinear Crossing Numbers of Certain Trees}\label{section crossing numbers}

	The most natural question one can ask about embedding graphs in the plane is if it can be done in such a way that produces no edge crossings. If so, this graph is called planar. If there is no way to embed a graph in the plane, the next natural question is to ask for the minimum number of edge crossings over all embeddings. This is known as the \emph{crossing number} of $G$, denoted $\mathrm{cr}(G)$ and has been extensively studied (\cite{Zarank},\cite{GJ},\cite{HH}). For a survey about crossing numbers and related problems, see \cite{S}.

	In addition to the standard crossing numbers, one can also compute the \emph{rectilinear crossing number} of a graph, which is the minimum number of crossings over all embeddings such that the edges are drawn as a straight line segment between the two incident vertices. Such a drawing is called a \emph{rectilinear drawing}. Rectilinear crossing numbers have received a lot of attention (\cite{J}, \cite{AF}, \cite{BD}, \cite{FPS}).
	
	Crossing numbers are usually defined in terms of the minimum number of crossings necessary, but an interesting maximization problem can also be studied. Mainly, one can ask for the maximum number of crossings possible over every rectilinear drawing. To study this problem, we will restrict ourselves to specific types of rectilinear drawings. A rectilinear drawing of a graph $G$ is a \emph{legal rectilinear drawing} if no edge passes through a vertex it is not incident with, and no three edges pass through the same point. The \emph{maximum rectilinear crossing number} of a graph $G$, denoted $\mathrm{max-\overline{cr}}(G)$ is the maximum number of crossings over all legal rectilinear drawings of $G$. We will only be concerned with legal drawings here, so henceforth we will assume all rectilinear drawings are legal.
		
	Standard (i.e. non-rectilinear) maximum crossing numbers have also been studied (\cite{CFKUVW}, \cite{HZ}, \cite{GR}, \cite{HM}). It is worth noting that while it is not known if standard maximum crossing numbers are monotone with respect to subgraphs, it is the case that maximum rectilinear crossing numbers are, or in other words if $F$ is a subgraph of $G$, then $\mathrm{max-\overline{cr}}(F)\leq \mathrm{max-\overline{cr}}(G)$ (\cite{RSP}). This fact will be very useful for us moving forward. For some of the known results about maximum rectilinear crossing numbers, see \cite{BJL}, \cite{F}, \cite{FHHK} and \cite{H}. The only known work on maximum rectilinear crossing numbers of trees is by Fallon, Hogenson, Keough, Lomel\'i, Schaefer and Sober\'on~\cite{Keo}.
	
	In this section we will present two results in maximum rectilinear crossing numbers involving certain classes of trees, both of which are solved using applications of Tur\'an's theorem and weighted Tur\'an numbers. The first result involves spiders, or trees that have a single vertex of degree $\geq 3$. This result iteratively uses Mantel's theorem (Tur\'ans theorem for the forbidden graph $K_3$). The second result involves trees of diameter at most $4$, and will follow from an application of the min-edge-weight Tur\'an numbers discussed in section \ref{section minedgeweight}. 
	First, we present a general observation about maximum crossing numbers
	
	\begin{observation}[The thrackle bound]
	Adjacent edges cannot cross eacho ther, so for any graph $G$,
	\[
	\mathrm{max-\overline{cr}}(G)\leq \binom{|E(G)|}2-\sum_{v\in V(G)}\binom{d(v)}2.
	\]
	\end{observation}

Graphs that attain the thrackle bound are called \emph{thrackles}. In \cite{W}, it was shown that in the rectilinear setting, caterpillars (trees such that if you remove all the leaves from the tree, you are left with a path) are thrackles, but no other trees are. Note that to get better upper bounds on $\mathrm{max-\overline{cr}}(G)$ we have to count pairs of nonadjacent edges that do not cross. We call such pairs {\em nontrivial missed crossings}.

	Now, we present a lemma that will be useful for both results. A \emph{spider} is a tree with a single vertex of degree $\geq 3$. Let $S$ be a spider and let $v$ be the vertex in $S$ with $d(v)\geq 3$. A maximal path that starts at $v$ is called a \emph{leg} of the spider. In each leg, the edge that is farthest from $v$ will be called a \emph{foot}, and a set of such edges will be called \emph{feet}. 
	
	\begin{lemma}\label{Lemma: three crossing edges} 
		Let $S$ be a spider with exactly three legs, each of length at least $2$. Let $e_1$, $e_2$, and $e_3$ be the three feet of $S$. Given a rectilinear drawing of $S$ where $e_1$, $e_2$ and $e_3$ all pairwise cross, then there is some edge $e\in E(S)\setminus \{e_1,e_2,e_3\}$ and some $1 \le i \le 3$ such that the pair $e$ and $e_i$ are a nontrivial missed crossing. Furthermore, the edges $e$ and $e_i$ do not appear on the same leg of $S$.
	\end{lemma}

\begin{proof}
	Assume $e_1$, $e_2$ and $e_3$ all cross in a drawing of $S$. Consider the plane separated into seven regions by the lines that pass through $e_1$, $e_2$ and $e_3$, and label the regions as in Figure \ref{figure:three edges crossing}, with the central region being labeled $C$.
	
	Let $v$ be the unique vertex of degree $3$ in $S$. We assume without loss of generality that $v$ appears in either $A_1$, $B_1$ or the central region, $C$. Let $v_1$ be the vertex of degree $2$ incident with $e_1$. Let us consider the path from $v$ to $v_1$. We need only find one edge that misses edges $e_2$ or $e_3$ once to satisfy the lemma. If to the contrary every edge intersects both $e_2$ and $e_3$, 
	then as we traverse the path from $v$ to $v_1$, the vertices must alternate between the region $A_1$ and region $B_1\cup C$. This gives a contradiction though since eventually this path must make it to $v_1$, which will require an edge to miss either $e_2$ or $e_3$. Whichever edge misses either $e_2$ or $e_3$, say without loss of generality $e_2$, is in a separate branch of $T''$ from $e_2$ since it is in the branch with $e_1$. This completes the proof.
	
\end{proof}

\begin{figure}[ht]
	\centering
	\begin{tikzpicture}[line cap=round,line join=round,>=triangle 45,x=1.0cm,y=1.0cm]
	\draw (0,0)-- (2,3);
	\draw (2,0)-- (0,3);
	\draw (-1,1.3)-- (3,1.3);
	\draw[dashed] (-1.9,1.3)--(3.9,1.3);
	\draw[dashed] (-0.5,-0.75)--(2.5,3.75);
	\draw[dashed] (-0.5,3.75)--(2.5,-0.75);

	\draw[color=black] (1,2.5) node {$A_1$};
	\draw[color=black] (1,.5) node {$B_1$};
	\draw[color=black] (-.5,.5) node {$A_2$};
	\draw[color=black] (2.5,2.5) node {$B_2$};
	\draw[color=black] (2.5,.5) node {$A_3$};
	\draw[color=black] (-.5,2.5) node {$B_3$};
	\draw[color=black] (-1,1.5) node {$e_1$};
		\draw[color=black] (-1,1.1) node {$v_1$};
	\draw[color=black] (0,3.2) node {$e_2$};
	\draw[color=black] (0,0.2) node {$e_3$};
	
	\draw [fill=black] (0,0) circle (2.0pt);
	\draw [fill=black] (2,0) circle (2.0pt);
	\draw [fill=black] (-1,1.3) circle (2.0pt);
	\draw [fill=black] (3,1.3) circle (2.0pt);
	\draw [fill=black] (0,3) circle (2.0pt);
	\draw [fill=black] (2,3) circle (2.0pt);
	\end{tikzpicture}
	\caption{\label{figure:three edges crossing} Areas}
\end{figure}

\subsection{Spiders}\label{section spiders}

In \cite{Keo}, the authors studied the maximum rectilinear crossing numbers of spiders. Given a spider $S$, let $v$ be the vertex of degree $\geq 3$. We will mainly be concerned with how many vertices are at distance $i$ from $v$ for each $i$, so we will say $S$ is a spider of type $(a_1,\dots,a_{\text{ecc}(v)})$ (Where $\text{ecc}(v)$ is the eccentricity of a vertex, or the length of the longest geodesic starting at $v$.) if there are exactly $a_i$ vertices at distance $i$ from $v$. The authors in \cite{Keo} describe an algorithm to draw spiders with many crossings which gives the following lower bound.

\begin{theorem}{\cite{Keo}, Proposition 2.2}\label{theorem spiders lower bound}
Let $S$ be a spider with $k\geq 3$ legs of length $\ell_1\geq\ell_2\geq\dots\geq\ell_k$. Then
\[
\mathrm{max-\overline{cr}}(S)\geq \binom{n-1}2-\sum_{v\in V(S)}\binom{d(v)}2-\sum_{i=3}^k(\ell_i-1)\left\lfloor\frac{i-1}2\right\rfloor.
\]
\end{theorem}

The authors further conjectured that this bound is correct. We confirm that conjecture by providing a matching upper bound. The following fact will be useful:
\begin{observation}\label{observation spider equality}
Let $S$ be a spider of type $(a_1,a_2,\dots,a_{\text{ecc}(V)})$ that has $k$ legs of length $\ell_1\geq \ell_2\geq \dots\geq \ell_k$. Then
\[
\sum_{i=3}^k(\ell_i-1)\left\lfloor\frac{i-1}2\right\rfloor=\sum_{j=2}^{\text{ecc}(v)}\left(\binom{\left\lfloor\frac{a_j}2\right\rfloor}2+\binom{\left\lceil\frac{a_j}2\right\rceil}2\right).
\]
\end{observation}

\begin{proof}
Let $v$ be the vertex of degree $\geq 3$ in $S$. First note that $a_1=k$ is the number of legs of $S$ that have length at least one, and in general $a_i$ is the number of legs of $S$ that have length at least $i$. The left-hand sum can be interpreted as counting edges that are not adjacent to $v$, and are not in the longest two legs of $S$, where each edge in the $i$th longest leg is given a weight of $\left\lfloor\frac{i-1}2\right\rfloor$. We can instead count this by adding up the weights of all the edges that are at the same distance from $v$, then summing over the possible distances. Thus,
\[
\sum_{i=3}^k(\ell_i-1)\left\lfloor\frac{i-1}2\right\rfloor=\sum_{j=2}^{\text{ecc}(v)}\sum_{i=1}^{a_j}\left\lfloor\frac{i-1}2\right\rfloor.
\]
Then, we have
\[
\sum_{j=2}^{\text{ecc}(v)}\sum_{i=1}^{a_j}\left\lfloor\frac{i-1}2\right\rfloor=\sum_{j=2}^{\text{ecc}(v)}\left(\sum_{s=1}^{\lfloor a_j/2\rfloor-1}s+\sum_{s=1}^{\lceil a_j/2\rceil-1}s\right)=\sum_{j=2}^{\text{ecc}(v)}\left(\binom{\left\lfloor\frac{a_j}2\right\rfloor}2+\binom{\left\lceil\frac{a_j}2\right\rceil}2\right).
\]
\end{proof}

\begin{theorem}
Given a spider $S$ of type  $(a_1,\dots,a_{\text{ecc}(v)})$,
\[ \mathrm{max}-\overline{cr}(S)=\binom{n-1}2-\sum_{v\in V(S)}\binom{d(v)}2-\sum_{i=2}^{\text{ecc}(v)}\left(\binom{\left\lfloor\frac{a_i}2\right\rfloor}2+\binom{\left\lceil\frac{a_i}2\right\rceil}2\right).
\]
\end{theorem}

\begin{proof}
	The lower bound follows from Theorem \ref{theorem spiders lower bound} and Observation \ref{observation spider equality}, so we need only establish the upper bound on $\mathrm{max-\overline{cr}(S)}$.	
	
	Fix a rectilinear drawing of $S$. Let $E_i$ be the set of the $a_i$ edges that connect the vertices at distance $i$ from $v$ to the vertices at distance $i-1$ from $v$. For each $2\leq i\leq \text{ecc}(v)$, we will construct an auxiliary graph $G_i$ with a vertex $v_e$ corresponding to each edge  $e \in E_i$, such that $v_e$ is adjacent to $v_f$ in $G_i$ if we have either {\em (i)} the edge $e$ does not cross some edge in the same leg of $S$ as $f$, or {\em (ii)} the edge $f$ does not cross some edge in the same leg of $S$ as $e$. For each triple $\{e_1=u_1v_1,\; e_2=u_2v_2,\; e_3=u_3 v_3\}\in\binom{E_i}3$, with $d(v_j,v)=i-1$ and $d(u_j,v)=i$ for $1\leq j\leq 3$), apply Lemma \ref{Lemma: three crossing edges} (Let $S$ be the spider with $3$ legs, each terminating at some $v_j$). Lemma \ref{Lemma: three crossing edges} tells us that $\{e_1,e_2,e_3\}$ does not induce a triangle in $G_i$. Thus by Mantel's theorem, this graph is missing at least $\binom{\left\lfloor a_i/2\right\rfloor}2+\binom{\left\lceil a_i/2\right\rceil}2$ edges. Each missing edge in $G_i$ corresponds to a pair of edges in $S$ forming a nontrivial missed crossing, at least one of which is in $E_i$.

	We first apply the above argument to $E_{\text{ecc}(v)}$, giving us $\binom{\left\lfloor a_{\text{ecc}(v)}/2\right\rfloor}2+\binom{\left\lceil a_{\text{ecc}(v)}/2\right\rceil}2$ missed crossings involving edges on this level. We can then delete all the edges in $E_{\text{ecc}(v)}$ and continue the process with the edges in $E_{\text{ecc}(v)-1}$. In this way, we do not overcount. Adding up all the missed crossings, we get that
		\[
		\maxcr(S)\leq \binom{n-1}2-\sum_{v\in V(S)}\binom{d(v)}2-\sum_{i=2}^{\text{ecc}(v)}\left(\binom{\left\lfloor\frac{a_i}2\right\rfloor}2+\binom{\left\lceil\frac{a_i}2\right\rceil}2\right).
		\]
	\end{proof}

\subsection{Trees of diameter $4$}\label{section diameter 4}

Now we present an application of the min-edge-weight Tur\'an numbers to maximum rectilinear crossing numbers. Note that if a tree has diameter $\leq 3$, it must be a caterpillar, and so attains the thrackle bound, as shown in \cite{W}. Thus, in this section we will assume the trees we are dealing with have diameter exactly $4$.

Given a tree $T$ of diameter $4$, let $v$ be the unique vertex of eccentricity $2$ in $T$. We will call $v$ the \emph{root} of $T$. Note that if there are only two vertices at distance $2$ from $v$, then $T$ is a caterpillar and so the crossing number is known. Thus, we are concerned with trees of height $2$ that have at least three vertices at distance $2$ from $v$. 

Given a tree $T$ of diameter $4$ with root $v$, we say $T$ is of type $(c_1,\dots,c_k)$, with $c_1\geq\dots\geq c_k$, if $v$ has $k$ children, $u_1,\dots,u_k$ such that $u_i$ has $c_i$ children for each $1\leq i\leq k$. Note that the type of a tree defined here is not the same as the type of a spider defined in Section \ref{section spiders}. Let 
\[
d=d(c_1,\dots,c_k)=\sum_{i=1}^{\lceil n/2\rceil-1} ic_{2i+1}+\sum_{i=1}^{\lfloor n/2\rfloor} ic_{2i+2},
\]
and note that given a vertex set $\{v_1,\dots,v_k\}$ with vertex weights $w(v_i)=c_i$, we have that
\[
d=w_{min}(K_k)-w_{min}(B_3(v_1,\dots,v_k))=w_{min}(K_k)-\ex(k,w_{min},K_3),
\]
where $B_3(v_1,\dots,v_n)$ is the graph defined in section \ref{section minedgeweight}, and the second equality follows from Proposition \ref{proposition minedgeweight}. It is worth noting that $B_3(v_1,\dots,v_n)$ is simply the balanced bipartite graph with even-indexed vertices in one partite set, and odd-indexed vertices in the other. We now show how to use min-edge-weight extremal numbers to bound the maximum rectilinear crossing numbers of trees of diameter $4$.

\begin{figure}[ht]
	\centering
	\begin{tikzpicture}[line cap=round,line join=round,>=triangle 45,x=1.0cm,y=2.0cm]\
	\draw[dashed] (-4,0)--(4,0);
	\draw[dashed] (-4,1)--(4,1);
	\draw (0,1)--(1.5,0);
	\draw (0,1)--(0.5,0);
	\draw (0,1)--(-0.5,0);
	\draw (0,1)--(-1.5,0);
	\draw (1.5,0)--(-2.1,1);
	\draw (1.5,0)--(-2.4,1);
	\draw (0.5,0)--(-1,1);
	\draw (-0.5,0)--(1.1,1);
	\draw (-0.5,0)--(1.4,1);
	\draw (-1.5,0)--(3.5,1);
	\draw (-1.5,0)--(2.5,1);
	\draw (-1.5,0)--(3,1);
	
	\draw[color=black] (-1.5,-0.2) node {$x_1$};
	\draw[color=black] (1.5,-0.2) node {$x_2$};
	\draw[color=black] (-0.5,-0.2) node {$x_3$};
	\draw[color=black] (0.5,-0.2) node {$x_4$};
	\draw[color=black] (0,1.2) node {$v$};
	\draw[color=black] (3,1.2) node {$x_1'$};
	\draw[color=black] (-2.25,1.2) node {$x_2'$};
	\draw[color=black] (1.25,1.2) node {$x_3'$};
	\draw[color=black] (-1,1.2) node {$x_4'$};

	\draw [fill=black] (0,1) circle (2.0pt);
	\draw [fill=black] (-1/2,0) circle (2.0pt);
	\draw [fill=black] (-1.5,0) circle (2.0pt);
	\draw [fill=black] (1/2,0) circle (2.0pt);
	\draw [fill=black] (1.5,0) circle (2.0pt);
	\draw [fill=black] (-1,1) circle (2.0pt);
	\draw [fill=black] (-2.4,1) circle (2.0pt);
	\draw [fill=black] (-2.1,1) circle (2.0pt);
	\draw [fill=black] (1.1,1) circle (2.0pt);
	\draw [fill=black] (1.4,1) circle (2.0pt);
	\draw [fill=black] (3,1) circle (2.0pt);
	\draw [fill=black] (2.5,1) circle (2.0pt);
	\draw [fill=black] (3.5,1) circle (2.0pt);
	\end{tikzpicture}
	\caption{\label{figure diameter 4 drawing} A tree of type (3,2,2,1)}
\end{figure}
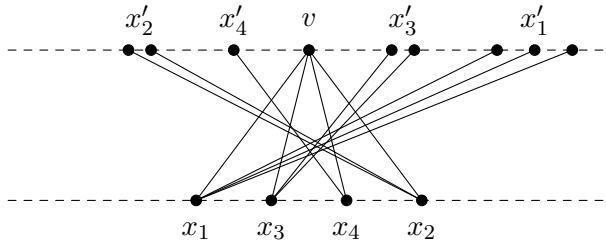

\begin{theorem}
	Let $T$ be a tree of diameter $4$ and type $(c_1,\dots,c_k)$, and let $d=d(c_1,\dots,c_k)$. Then 
	\[
	\maxcr(T)=\binom{n-1}2-\sum_{v\in V(T)}\binom{d(v)}2-d.
	\]
\end{theorem}

\begin{proof}
	First let us describe a drawing of $T$ that misses only $d$ nontrivial crossings. In the canonical $xy$-plane, Our vertices will appear either on the line $y=0$ or $y=1$, and every edge will be incident with vertices on both lines. Let $v$ be the root of $T$ and $v_1,\dots,v_k$ be the children of $v$. For each $1\leq i\leq k$, let $v_{i,1},\dots,v_{i,c_i}$ be the $c_i$ children of $v_i$.
	
	We define the following useful tool for drawings of trees. Suppose $uv$ and $wv$ are adjacent edges in $T$. We say that $uv$ is a {\em clone} of $wv$ in a particular drawing of $T$ if $uv$ and $wv$ cross precisely the same set of edges. Note that if we are given a drawing of $T$, we can alter the drawing to make $uv$ a clone of $wv$ simply by embedding $u$ very close to $w$. 
	
	Let $v$ be embedded at $(0,1)$. Now embed $v_i$ at $(x_i, 0)$ for some numbers (say integers) $x_i$ such that {\em (i)} $x_i < x_j$ when $i$ is odd and $j$ is even, {\em (ii)} $x_i < x_j$ when $i, j$ are even and $i < j$, and {\em (iii)} $x_i > x_j$ when $i, j$ are odd and $i < j$. In other words, $x_1 < x_3 < x_5 \ldots < x_6 < x_4 < x_2$. Finally, we will embed the grandchildren $v_{i, j}$ on the line $y=1$. Specifically, for each odd $i$ we will choose a positive integer $x_i'$ and for each even $i$ we will choose a negative integer $x_i'$, such that $x_2' < x_4' < \ldots < 0$ and $x_1' > x_3' \ldots > 0$. We now embed all of the grandchildren $v_{i, j}$ very close to the point $(x_i', 1)$, so that they are all clones of each other. For an example, see Figure \ref{figure diameter 4 drawing}.

	Now, note that in this drawing of $T$, each nontrivial missed crossing occurs between an edge $v v_i$ and an edge $v_{i'}v_{i',j'}$ for some $i,i',j'$ with $i$ and $i'$ of the same parity, and $i<i'$. Furthermore, whenever there exists such a missed crossing, then $v v_i$ misses every edge $v_{i'}v_j$ for all $1\leq j\leq c_{i'}$. Notice that for $v_{2i+1}$, with $i\geq 1$, there are exactly $i$ vertices $v_j$ with $j$ odd and $j<2i+1$, so the edges from $v_{2i+1}$ to its grandchildren account for $ic_{2i+1}$ missed crossings. Similarly, each vertex $v_{2i+2}$, with $i\geq 1$, accounts for $ic_{2i+2}$ missed crossings. This gives us exactly $d$ nontrivial missed crossings, so $\maxcr(T)\geq \binom{n-1}{2}-\sum_{v\in V(T)}\binom{d(v)}2-d$.
	
	We now will focus on the upper bound for $\maxcr(T)$. Let $\mathcal{D}$ be a drawing of $T$ in the plane with $\maxcr(T)$ crossings. We may assume that for each $i$, all the edges $v_i v_{i, j}$ are clones of each other. To see this, fix $i$ and let $j^*$ be the index such that $v_iv_{i,j^*}$ crosses the largest number of other edges. Then we may alter the drawing of $T$ to make all edges $v_iv_{i,j}$ into clones of $v_iv_{i,j^*}$ which can only increase the number of crossings.
	
	Let us consider an auxiliary weighted graph $H$ on vertex set $V(H)=\{v_1,\dots,v_k\}$ with vertex weights $w(v_i)=c_i$ and edge weights given by $w_{min}$. The edge $v_iv_j$ will be included in $H$ if and only if in the drawing of $T$ the edge $v_iv_{i,1}$ crosses both edges $vv_j$ and $v_jv_{j,1}$, and the edge $vv_i$ crosses $v_jv_{j,1}$. If we apply Lemma \ref{Lemma: three crossing edges} to the edges $v_iv_{i,1}$, $v_{i'}v_{i',1}$ and $v_{i''}v_{i'',1}$ ($S$ will be the $3$-legged spider with these edges as feet), we see that the vertices in \{$v_i$, $v_{i'}$, $v_{i''}$\} do not induce a triangle in $H$. Thus $H$ is triangle-free. Note that any non-edge, say $v_iv_{i'}$ in $H$, corresponds to a missed crossing involving either $v_iv_{i,1}$ or $v_{i'}v_{i',1}$, say $v_iv_{i,1}$. Since in our drawing, $v_{i,j}$ misses every edge that $v_{i,1}$ does for each $1\leq j\leq c_i$, this non-edge in $H$ is in one-to-one correspondence with $c_i$ missed crossings in our drawing that are not counted by the thrackle bound. Thus, every non-edge $v_iv_{i'}$ contributes at least $w_{min}(v_iv_{i'})$ missed crossings, and so we have at least $d=w_{min}(K_k)-\ex(k,w_{min},K_3)$ nontrivial missed crossings, finishing the proof.
\end{proof}

\section{Conclusion and open problems}

For the Tur\'an problem in a multipartite host graph, of course the next question to ask is the value of $\ex(K_{k_1, k_2, \ldots, k_r}, s K_\l)$ for $s \ge 2$. A naive approach using the vertex-duplication operation fails. Indeed it is possible that $G$ is $s K_\l$-free, but duplicating a vertex makes a new graph that is not. 

It may also be of interest to find an efficient algorithm that, given the parameters $\l, k_1, \ldots k_r$, determines the value of the maximum in Theorem \ref{thm:turan}. The authors have not made a serious attempt to find such an algorithm, but it is possible that none exists. Indeed, this problem is similar to the well-studied load-balancing problem which is known to be NP-hard. 

For rectilinear crossing numbers, the next natural question to ask is for trees of larger diameter. However, even for general diameter 5 trees the problem seems to get significantly more complicated, although some of our arguments do still apply.

It may also be of some interest to study more vertex-induced edge weightings. The next most natural problem along these lines may be the sum-edge-weighting $w_+$ given by $w_+(xy)=W(x)+W(y)$. 

\bibliographystyle{amsplain}
\bibliography{refs}

\end{document}